\documentclass[11pt]{article}
\usepackage[a4paper,margin=1in]{geometry}
\usepackage[T1]{fontenc}
\usepackage[utf8]{inputenc}
\usepackage{lmodern}
\usepackage{amsmath,amssymb,amsthm,mathtools}
\usepackage{enumitem}
\usepackage{hyperref}
\usepackage[nameinlink,capitalise,noabbrev]{cleveref}
\usepackage{microtype}
\usepackage{xcolor}
\usepackage{booktabs}
\usepackage{array}
\usepackage{longtable}
\usepackage{setspace}
\usepackage{mathrsfs}

\hypersetup{
  colorlinks=true,
  linkcolor=blue!50!black,
  citecolor=blue!50!black,
  urlcolor=blue!50!black,
  pdftitle={Simplex Layers and Phase Boundaries in the Partition Graph},
  pdfauthor={Fedor B. Lyudogovskiy}
}

\newtheorem{theorem}{Theorem}[section]
\newtheorem{proposition}[theorem]{Proposition}
\newtheorem{lemma}[theorem]{Lemma}
\newtheorem{corollary}[theorem]{Corollary}
\theoremstyle{definition}

\newtheorem{problem}[theorem]{Problem}
\theoremstyle{remark}
\newtheorem{remark}[theorem]{Remark}

\newcommand{\Par}{\mathrm{Par}}
\newcommand{\Cl}{\mathrm{Cl}}
\newcommand{\dimloc}{\mathrm{dim}_{\mathrm{loc}}}
\newcommand{\SpecDelta}{\mathrm{Spec}_{\Delta}}
\newcommand{\Deltaloc}{\Delta_{\mathrm{loc}}}
\newcommand{\deltaloc}{\delta_{\mathrm{loc}}}
\newcommand{\omegaloc}{\omega_{\mathrm{loc}}}
\newcommand{\first}{\mathrm{first}}
\newcommand{\Fset}{\mathcal{F}}
\newcommand{\Ebdry}{\partial^{E}}
\newcommand{\Vbdry}{\partial^{V}}
\newcommand{\bminus}{\partial^{-}}
\newcommand{\bplus}{\partial^{+}}
\newcommand{\toplayer}{L_{\mathrm{top}}}
\newcommand{\Stair}{\mathcal{S}}

\title{Simplex Layers and Phase Boundaries in the Partition Graph}
\author{Fedor B. Lyudogovskiy}
\date{}

\begin{document}
\maketitle

\begin{abstract}
For the partition graph $G_n$ on the set of partitions of $n$, we study the stratification induced by the local simplex dimension
\[
\dimloc(\lambda),\qquad \lambda\vdash n,
\]
defined as the maximal dimension of a simplex of the clique complex $K_n=\Cl(G_n)$ containing $\lambda$. Using the previously established description of maximal cliques through a vertex in terms of star and top capacities, we define the simplex layers
\[
L_r(n):=\{\lambda\vdash n:\dimloc(\lambda)=r\}
\]
and investigate their global structure.

We formalize the layer stratification of $G_n$, rewrite layer membership in terms of the local star- and top-capacities, and record its basic structural consequences, including conjugation invariance. We then study first occurrence of layers across $n$, introducing the first-occurrence indices $n_r^{\first}$ and the corresponding first-occurrence sets $\Fset_r$. For the initial layer values, this yields explicit theorem-level results; more generally, it leads to a first-occurrence table and several natural sequence questions.

We also define a canonical graph-theoretic notion of phase boundary between adjacent layers, namely the adjacent-layer edge boundary
\[
\Ebdry_{r,r+1}(n),
\]
consisting of edges joining $L_r(n)$ to $L_{r+1}(n)$. This gives an interface language for the layer stratification, together with one-sided and vertex-boundary variants. On the computational side, we record the initial first-occurrence table, note a finite-range computational pattern for several subsequent layer values, and indicate the broader datasets naturally attached to the stratification, including layer-size, top-layer, boundary, and restricted-count data.

The paper treats the simplex layers as exact level sets of a local invariant, distinct from the broader shell-type geometric language used elsewhere in the project. Thus the layer stratification and its phase boundaries provide a formal skeleton for the simplicial geometry of the partition graph.
\end{abstract}



\section{Introduction}

The partition graph $G_n$ is the graph on $\Par(n)$ whose edges correspond to elementary transfers of one unit between parts, followed by reordering. Several aspects of this graph have been studied earlier in this series, including local clique structure \cite{ProjectLocal}, axial and directional morphology \cite{ProjectAxial,ProjectDirectional}, boundary geometry \cite{ProjectBoundary}, shell geometry \cite{ProjectShells}, and the topology of the clique complex. In the present paper we focus on one bridge between the local and global theories. Namely, we study the stratification of $G_n$ by local simplex dimension.

Let
\[
K_n:=\Cl(G_n)
\]
be the clique complex of $G_n$. For a vertex $\lambda\in \Par(n)$, its local simplex dimension
\[
\dimloc(\lambda)
\]
is the maximal dimension of a simplex of $K_n$ containing $\lambda$. Earlier work in the project established that this invariant is governed by two canonical families of simplices, namely star-simplices and top-simplices, and that the corresponding local capacities determine $\dimloc$. This local theory is the main input for the present paper.

This leads to the exact simplex layers
\[
L_r(n):=\{\lambda\vdash n:\dimloc(\lambda)=r\},
\]
which are the main objects of the paper. These layers are level sets of a local invariant and therefore define a canonical stratification of $\Par(n)$. Our goal is to study this stratification systematically: to describe its basic formal properties, to investigate the first occurrence of layers as $n$ grows, and to define the exact interfaces between adjacent layers.

The first motivation for this study is structural. The local simplex dimension is one of the most natural local simplicial invariants of the partition graph, and the passage from the local formula for $\dimloc$ to the global layer structure of $G_n$ is mathematically natural in its own right. It yields level sets, top layers, first-occurrence questions, and phase boundaries between neighboring layer values.

The second motivation is geometric. Elsewhere in the project, broader shell-type language has been used to describe the global morphology of $G_n$, including triangular skin, tetrahedral core, higher-dimensional thickening, and rear-central thickening; see, for example, \cite{ProjectBoundary,ProjectShells}. The present paper does not replace that geometric language; rather, it sharpens one part of it. The simplex layers are not broad geometric zones; they are exact level sets of a local invariant. They therefore provide a more rigid formal skeleton underlying part of the previously observed global morphology.

The distinction between these two languages is important. In this paper, statements about
\[
L_r(n),\qquad \SpecDelta(n),\qquad \Deltaloc(n),\qquad \Ebdry_{r,r+1}(n)
\]
are exact layer-theoretic statements. By contrast, statements about skins, cores, framework regions, or rear-central zones belong to a coarser geometric description. The two descriptions are related, but they are not identical.

The main contributions of the paper are the following.

First, we formalize the layer stratification induced by $\dimloc$. Using the local star/top dimension formula, we rewrite layer membership in terms of local capacities and record the basic consequences: the layers partition $\Par(n)$, the realized spectrum is defined, top layers are nonempty, and all layers are invariant under partition conjugation.

Second, we organize the first-occurrence problem for layers. For each $r$, we consider the smallest $n$ for which the layer $L_r(n)$ is nonempty and the set of partitions that realize layer $r$ at that minimal value of $n$. For the initial layer values this yields explicit theorem-level results; more generally, it leads to a first-occurrence table and a family of sequence questions.

Third, we define a formal notion of phase boundary between adjacent layers. The primary notion adopted here is the adjacent-layer edge boundary
\[
\Ebdry_{r,r+1}(n),
\]
consisting of those edges of $G_n$ whose endpoints lie in the two neighboring exact layers. This gives a graph-theoretic interface between $L_r(n)$ and $L_{r+1}(n)$, together with associated one-sided and vertex-boundary notions.

Fourth, we record the initial first-occurrence table explicitly, note a finite-range computational pattern for several subsequent layer values, and describe the broader computational framework attached to the stratification. This includes layer-size data, top-layer data, adjacent-layer boundary counts, and restricted counts inside selected geometric regions. Such computations do not replace exact proofs, but they indicate the natural datasets associated with the layer formalism and help separate theorem-level facts from broader empirical patterns.

The present paper is devoted to the stratification induced by the local simplex dimension. Its subject is not the full shell geometry of the partition graph, and not the topology of the clique complex as such, but the level-set structure produced by $\dimloc$. It therefore occupies an intermediate position in the broader cycle: it converts the local simplicial theory into a global vertex stratification and introduces the corresponding boundary language.

The paper is organized as follows. In Section~\ref{sec:preliminaries} we recall the partition graph, the clique complex, and the imported local dimension formula based on star/top capacities. In Section~\ref{sec:stratification} we define the simplex layers and establish their basic formal properties. Section~\ref{sec:firstocc} studies the first occurrence of layers. Section~\ref{sec:boundaries} introduces the adjacent-layer phase boundaries. Section~\ref{sec:localization} develops a restriction-based localization framework and compares the exact layer language with the coarser shell geometry. Section~\ref{sec:atlas} briefly describes the natural computational datasets and associated sequences. Section~\ref{sec:conclusion} concludes with a summary and a list of open problems.

\section{Preliminaries and the local dimension formula}\label{sec:preliminaries}

\subsection{The partition graph and the clique complex}

Let $n\ge 1$, and let $G_n$ be the partition graph on $\Par(n)$, whose vertices are the partitions of $n$, with adjacency given by an elementary transfer of one unit between two parts, followed by reordering.

We write
\[
K_n:=\Cl(G_n)
\]
for the clique complex of $G_n$.

\subsection{Local simplex dimension}

For a vertex $\lambda\in \Par(n)$, its local simplex dimension is defined by
\[
\dimloc(\lambda):=\max\{\dim \sigma:\sigma \text{ is a simplex of }K_n\text{ and }\lambda\in \sigma\}.
\]
Equivalently, if $\omegaloc(\lambda)$ denotes the maximum size of a clique in $G_n$ containing $\lambda$, then
\[
\dimloc(\lambda)=\omegaloc(\lambda)-1.
\]

The local simplicial structure through a fixed vertex is governed by two canonical families of simplices: star-simplices and top-simplices. These were developed earlier in the project and will be used here only through the associated local capacities.

\subsection{Star- and top-capacities}

Let $\lambda\vdash n$. A transfer $\lambda(c\to a)$ is called admissible if moving one cell from the removable corner $c$ to the addable corner $a$, followed by reordering, again yields a partition of $n$. For a removable corner $c$ of $\lambda$, let
\[
A_{\max}(\lambda,c)
\]
be the set of addable corners $a$ such that the transfer $\lambda(c\to a)$ is admissible. For an addable corner $a$ of $\lambda$, let
\[
C_{\max}(\lambda,a)
\]
be the set of removable corners $c$ such that the same transfer is admissible.

Define the star-capacity and top-capacity of $\lambda$ by
\[
s(\lambda):=\max_c |A_{\max}(\lambda,c)|,
\qquad
t(\lambda):=\max_a |C_{\max}(\lambda,a)|.
\]

Thus $s(\lambda)$ measures the largest star-simplex through $\lambda$, while $t(\lambda)$ measures the largest top-simplex through $\lambda$.

\subsection{Imported local dimension formula}

We record the local dimension formula established in \cite[Proposition~5.4]{ProjectLocal}.

\begin{theorem}[Imported local dimension formula]\label{thm:imported_local_dimension}
For every partition $\lambda\vdash n$,
\[
\dimloc(\lambda)=\max\{s(\lambda),t(\lambda)\}.
\]
\end{theorem}

In the present paper we adopt this normalization throughout.

As an immediate consequence, for every $\lambda\vdash n$ and every $r\ge 0$,
\[
\dimloc(\lambda)\ge r
\iff
s(\lambda)\ge r \text{ or } t(\lambda)\ge r,
\]
\[
\dimloc(\lambda)\le r
\iff
s(\lambda)\le r \text{ and } t(\lambda)\le r,
\]
and
\[
\dimloc(\lambda)=r
\iff
\max\{s(\lambda),t(\lambda)\}=r.
\]

From this point on, the earlier local theory is used only through the quantities $s(\lambda)$, $t(\lambda)$, and Theorem~\ref{thm:imported_local_dimension}.

\subsection{Basic notation for layers and first occurrences}

For $r\ge 0$, define the $r$-th simplex layer by
\[
L_r(n):=\{\lambda\vdash n:\dimloc(\lambda)=r\}.
\]

We also define the realized layer spectrum
\[
\SpecDelta(n):=\{r\ge 0:L_r(n)\neq\varnothing\},
\]
the top local simplex dimension
\[
\Deltaloc(n):=\max_{\lambda\vdash n}\dimloc(\lambda),
\]
and the corresponding top layer
\[
\toplayer(n):=L_{\Deltaloc(n)}(n).
\]

For later use, we define the first-occurrence index
\[
n_r^{\first}:=\min\{n\ge 1:L_r(n)\neq\varnothing\},
\]
whenever this minimum exists, and the corresponding first-occurrence set
\[
\Fset_r:=L_r(n_r^{\first}).
\]

Finally, for adjacent layers we will use the adjacent-layer edge boundary
\[
\Ebdry_{r,r+1}(n)
:=
\bigl\{
\{\lambda,\mu\}\in E(G_n):
\lambda\in L_r(n),\ \mu\in L_{r+1}(n)
\bigr\}.
\]

\section{Simplex layers and exact stratification}\label{sec:stratification}

\subsection{Exact layer decomposition}

Recall that
\[
L_r(n):=\{\lambda\vdash n:\dimloc(\lambda)=r\}.
\]

These are the level sets of the integer-valued function
\[
\dimloc:\Par(n)\to \mathbb Z_{\ge 0}.
\]

\begin{proposition}\label{prop:partition}
For every $n\ge 1$, the family of nonempty layers
\[
\{L_r(n):r\in \SpecDelta(n)\}
\]
forms a partition of $\Par(n)$. More precisely:
\begin{enumerate}[label=(\arabic*)]
\item $L_r(n)\cap L_s(n)=\varnothing$ whenever $r\neq s$;
\item \[
   \Par(n)=\bigsqcup_{r\in \SpecDelta(n)}L_r(n).
   \]
\end{enumerate}
\end{proposition}

\begin{proof}
Both statements are immediate from the definition of $L_r(n)$ as the exact level set of $\dimloc$. \qedhere
\end{proof}

The decomposition itself is tautological. The substantive questions concern realizability, first occurrence, and adjacency between layers.

\subsection{Layer criteria in terms of local capacities}

The local dimension formula yields direct criteria for layer membership.

\begin{proposition}\label{prop:criteria}
Let $\lambda\vdash n$, and let $r\ge 0$. Then:
\begin{enumerate}[label=(\arabic*)]
\item $\lambda\in L_r(n)$ if and only if
   \[
   \max\{s(\lambda),t(\lambda)\}=r;
   \]

\item $\dimloc(\lambda)\ge r$ if and only if
   \[
   s(\lambda)\ge r \text{ or } t(\lambda)\ge r;
   \]

\item $\dimloc(\lambda)\le r$ if and only if
   \[
   s(\lambda)\le r \text{ and } t(\lambda)\le r.
   \]
\end{enumerate}
\end{proposition}

\begin{proof}
Immediate from Theorem~\ref{thm:imported_local_dimension}. \qedhere
\end{proof}

\begin{corollary}\label{cor:criteria_exact}
A partition $\lambda\vdash n$ lies in $L_r(n)$ if and only if
\[
s(\lambda)\le r,\qquad t(\lambda)\le r,
\]
and at least one of the equalities
\[
s(\lambda)=r,\qquad t(\lambda)=r
\]
holds.
\end{corollary}

\begin{proof}
This is just a restatement of
\[
\max\{s(\lambda),t(\lambda)\}=r.
\]
\end{proof}

Thus each layer may be viewed as the locus where the pair $(s(\lambda),t(\lambda))$ reaches the value $r$ but does not exceed it.

\subsection{Realized spectrum, bottom layer, and top layer}

Recall that
\[
\SpecDelta(n):=\{r\ge 0:L_r(n)\neq\varnothing\},
\qquad
\Deltaloc(n):=\max_{\lambda\vdash n}\dimloc(\lambda).
\]

We also define the bottom local simplex dimension
\[
\deltaloc(n):=\min_{\lambda\vdash n}\dimloc(\lambda).
\]

\begin{proposition}\label{prop:extremal_layers}
For every $n\ge 1$,
\[
\Deltaloc(n)=\max \SpecDelta(n),
\qquad
\deltaloc(n)=\min \SpecDelta(n),
\]
and the top layer satisfies
\[
\toplayer(n)=L_{\Deltaloc(n)}(n)\neq\varnothing.
\]
\end{proposition}

\begin{proof}
By definition,
\[
\Deltaloc(n)=\max_{\lambda\vdash n}\dimloc(\lambda),
\qquad
\deltaloc(n)=\min_{\lambda\vdash n}\dimloc(\lambda).
\]
Hence both values are attained by some partitions of $n$, so both belong to $\SpecDelta(n)$. Every realized layer value is attained as $\dimloc(\lambda)$ for some $\lambda\vdash n$, and is therefore bounded between $\deltaloc(n)$ and $\Deltaloc(n)$. This proves the two identities, and nonemptiness of $\toplayer(n)$ is immediate. \qedhere
\end{proof}

\begin{proposition}\label{prop:spectrum_interval}
For every $n\ge 1$,
\[
\SpecDelta(n)\subseteq [\deltaloc(n),\Deltaloc(n)]\cap \mathbb Z.
\]
Equality holds precisely when every integer between $\deltaloc(n)$ and $\Deltaloc(n)$ is realized.
\end{proposition}

\begin{proof}
The inclusion is immediate from Proposition~\ref{prop:extremal_layers}. Equality is equivalent to saying that every integer between the minimum and maximum realized layer values is itself realized.\qedhere
\end{proof}

The two basic numerical layer statistics are therefore the maximal realized local simplex dimension $\Deltaloc(n)$ and the size of the top layer $|\toplayer(n)|$.

\subsection{Conjugation invariance of layers}

Let $\lambda\mapsto \lambda'$ denote partition conjugation. Since conjugation is an automorphism of $G_n$, it preserves clique structure and therefore preserves local simplex dimension.

\begin{proposition}\label{prop:conj_dimloc}
For every $\lambda\vdash n$,
\[
\dimloc(\lambda')=\dimloc(\lambda).
\]
\end{proposition}

\begin{proof}
Conjugation is an automorphism of $G_n$, and therefore induces an automorphism of $K_n$. It carries the simplices containing $\lambda$ bijectively onto the simplices containing $\lambda'$, so the maximal simplex dimension through $\lambda$ is preserved.\qedhere
\end{proof}

\begin{corollary}\label{cor:conj_layers}
For every $n\ge 1$ and every $r\ge 0$, the layer $L_r(n)$ is invariant under conjugation:
\[
\lambda\in L_r(n)\quad\Longleftrightarrow\quad \lambda'\in L_r(n).
\]
\end{corollary}

\begin{proof}
Immediate from Proposition~\ref{prop:conj_dimloc} and the definition of $L_r(n)$.\qedhere
\end{proof}

\begin{corollary}\label{cor:pairing_layers}
If $L_r(n)$ contains a non-self-conjugate partition, then it contains that partition together with its conjugate. In particular, every non-self-conjugate contribution to a layer occurs in conjugate pairs.
\end{corollary}

\begin{proof}
Immediate from Corollary~\ref{cor:conj_layers}.\qedhere
\end{proof}

\subsection{First formal consequences}

We record two immediate consequences of the exact layer formalism.

First, whenever the layer $L_r$ is realized for some $n$, the corresponding first-occurrence index
\[
n_r^{\first}=\min\{n\ge 1:L_r(n)\neq\varnothing\}
\]
is well defined.

Second, the layer profile array
\[
a_{n,r}:=|L_r(n)|
\]
encodes the basic stratification data for fixed $n$: it determines the realized spectrum $\SpecDelta(n)$, the extremal layer values $\deltaloc(n)$ and $\Deltaloc(n)$, and the size of the top layer.

These observations will be used in two different directions. Section~\ref{sec:firstocc} studies the first appearance of layers across $n$, while later sections examine the adjacency structure between layers and the computational framework naturally associated with the resulting stratification.

\section{First occurrence of layers}\label{sec:firstocc}

\subsection{First-occurrence indices and first-occurrence sets}

We study the first appearance of simplex layers along the sequence $(G_n)_{n\ge 1}$. For a fixed layer value $r$, the basic objects are the first-occurrence index
\[
n_r^{\first}:=\min\{n\ge 1:L_r(n)\neq\varnothing\},
\]
whenever this minimum exists, and the corresponding first-occurrence set
\[
\Fset_r:=L_r(n_r^{\first}).
\]

Thus $n_r^{\first}$ records the earliest value of $n$ for which the local simplex dimension $r$ is realized, while $\Fset_r$ records the partitions that realize it for the first time.

For the present paper, the first-occurrence problem is treated at two levels. For the initial layer values, it leads to exact theorem-level results. More generally, it leads to a computational first-occurrence table and to a natural family of sequence questions.

\subsection{Criterion for first occurrence}

By Proposition~\ref{prop:criteria}, the existence of $L_r(n)$ is equivalent to the existence of a partition $\lambda\vdash n$ such that
\[
\max\{s(\lambda),t(\lambda)\}=r.
\]
Equivalently, $n_r^{\first}$ is the least positive integer $n$ for which there exists $\lambda\vdash n$ satisfying
\[
s(\lambda)\le r,\qquad t(\lambda)\le r,
\]
and at least one of the equalities
\[
s(\lambda)=r,\qquad t(\lambda)=r
\]
holds.

Thus the first-occurrence problem becomes an extremal problem for the local star- and top-capacities.

It is also convenient to distinguish two basic candidate mechanisms for first occurrence:
\begin{enumerate}[label=(\arabic*)]
\item star-first candidates, for which
   \[
   s(\lambda)=r,\qquad t(\lambda)\le r;
   \]

\item top-first candidates, for which
   \[
   t(\lambda)=r,\qquad s(\lambda)\le r.
   \]
\end{enumerate}

These two mechanisms are dual under conjugation. The paper does not assume a priori that every first occurrence is purely star-driven or purely top-driven.

\subsection{Exact first occurrences for the initial layers}

For $r\ge 1$, let
\[
\delta_r:=(r,r-1,\dots,2,1),
\qquad
|\delta_r|=\frac{r(r+1)}{2},
\]
be the staircase partition of size $\frac{r(r+1)}{2}$.
Let $\Stair_r$ denote the set of partitions obtained from $\delta_r$ by adding a single cell at an addable corner.
For convenience, also set
\[
\delta_0:=\varnothing,
\qquad
\Stair_0:=\{(1)\}.
\]
Then $|\Stair_r|=r+1$ for every $r\ge 0$, since the staircase partition $\delta_r$ has exactly $r+1$ addable corners.

The initial exact first-occurrence data are as follows.

\begin{theorem}\label{thm:firstocc_initial}
For $r=0,1,2,3$, one has
\[
n_r^{\first}=1+\frac{r(r+1)}{2},
\qquad
\Fset_r=\Stair_r.
\]
Explicitly,
\[
\Fset_0=\{(1)\},
\]
\[
\Fset_1=\{(2),(1,1)\},
\]
\[
\Fset_2=\{(3,1),(2,2),(2,1,1)\},
\]
and
\[
\Fset_3=\{(4,2,1),(3,3,1),(3,2,2),(3,2,1,1)\}.
\]
In particular,
\[
n_0^{\first}=1,\qquad
n_1^{\first}=2,\qquad
n_2^{\first}=4,\qquad
n_3^{\first}=7.
\]
\end{theorem}

Table~\ref{tab:smallverify} records the complete finite verification underlying Theorem~\ref{thm:firstocc_initial}. It lists the values of $s(\lambda)$, $t(\lambda)$, and $\dimloc(\lambda)=\max\{s(\lambda),t(\lambda)\}$ for every partition $\lambda\vdash n$ with $1\le n\le 7$.

\begin{center}
\footnotesize
\setlength{\LTleft}{0pt}
\setlength{\LTright}{0pt}
\begin{longtable}{c l c c c}
\caption{Complete verification data for partitions of size at most $7$.}\label{tab:smallverify}\\
\toprule
$n$ & $\lambda$ & $s(\lambda)$ & $t(\lambda)$ & $\dimloc(\lambda)$ \\
\midrule
\endfirsthead
\toprule
$n$ & $\lambda$ & $s(\lambda)$ & $t(\lambda)$ & $\dimloc(\lambda)$ \\
\midrule
\endhead
\midrule
1 & $(1)$ & 0 & 0 & 0 \\
\midrule
2 & $(2)$ & 1 & 1 & 1 \\
 & $(1,1)$ & 1 & 1 & 1 \\
\midrule
3 & $(3)$ & 1 & 1 & 1 \\
 & $(2,1)$ & 1 & 1 & 1 \\
 & $(1,1,1)$ & 1 & 1 & 1 \\
\midrule
4 & $(4)$ & 1 & 1 & 1 \\
 & $(3,1)$ & 2 & 1 & 2 \\
 & $(2,2)$ & 2 & 1 & 2 \\
 & $(2,1,1)$ & 2 & 1 & 2 \\
 & $(1,1,1,1)$ & 1 & 1 & 1 \\
\midrule
5 & $(5)$ & 1 & 1 & 1 \\
 & $(4,1)$ & 2 & 1 & 2 \\
 & $(3,2)$ & 2 & 2 & 2 \\
 & $(3,1,1)$ & 2 & 2 & 2 \\
 & $(2,2,1)$ & 2 & 2 & 2 \\
 & $(2,1,1,1)$ & 2 & 1 & 2 \\
 & $(1,1,1,1,1)$ & 1 & 1 & 1 \\
\midrule
6 & $(6)$ & 1 & 1 & 1 \\
 & $(5,1)$ & 2 & 1 & 2 \\
 & $(4,2)$ & 2 & 2 & 2 \\
 & $(4,1,1)$ & 2 & 2 & 2 \\
 & $(3,3)$ & 2 & 1 & 2 \\
 & $(3,2,1)$ & 2 & 2 & 2 \\
 & $(3,1,1,1)$ & 2 & 2 & 2 \\
 & $(2,2,2)$ & 2 & 1 & 2 \\
 & $(2,2,1,1)$ & 2 & 2 & 2 \\
 & $(2,1,1,1,1)$ & 2 & 1 & 2 \\
 & $(1,1,1,1,1,1)$ & 1 & 1 & 1 \\
\midrule
7 & $(7)$ & 1 & 1 & 1 \\
 & $(6,1)$ & 2 & 1 & 2 \\
 & $(5,2)$ & 2 & 2 & 2 \\
 & $(5,1,1)$ & 2 & 2 & 2 \\
 & $(4,3)$ & 2 & 2 & 2 \\
 & $(4,2,1)$ & 3 & 2 & 3 \\
 & $(4,1,1,1)$ & 2 & 2 & 2 \\
 & $(3,3,1)$ & 3 & 2 & 3 \\
 & $(3,2,2)$ & 3 & 2 & 3 \\
 & $(3,2,1,1)$ & 3 & 2 & 3 \\
 & $(3,1,1,1,1)$ & 2 & 2 & 2 \\
 & $(2,2,2,1)$ & 2 & 2 & 2 \\
 & $(2,2,1,1,1)$ & 2 & 2 & 2 \\
 & $(2,1,1,1,1,1)$ & 2 & 1 & 2 \\
 & $(1,1,1,1,1,1,1)$ & 1 & 1 & 1 \\
\bottomrule
\end{longtable}
\end{center}

\begin{proof}
By Table~\ref{tab:smallverify}, the unique partition of $1$ has local simplex dimension $0$, every partition of $2$ or $3$ has local simplex dimension $1$, exactly the three partitions
\[
(3,1),\qquad (2,2),\qquad (2,1,1)
\]
of $4$ have local simplex dimension $2$, and exactly the four partitions
\[
(4,2,1),\qquad (3,3,1),\qquad (3,2,2),\qquad (3,2,1,1)
\]
of $7$ have local simplex dimension $3$.

The same table shows that no partition of size at most $3$ has local simplex dimension $2$, and that no partition of size at most $6$ has local simplex dimension $3$. Therefore
\[
n_0^{\first}=1,\qquad n_1^{\first}=2,\qquad n_2^{\first}=4,\qquad n_3^{\first}=7,
\]
and the corresponding first-occurrence sets are exactly the sets listed in the statement. Since these sets coincide with $\Stair_r$ for $r=0,1,2,3$, the claim follows.\qedhere
\end{proof}

\subsection{Symmetry of first-occurrence sets}

The first-occurrence sets inherit the conjugation symmetry of the layers.

\begin{proposition}\label{prop:firstocc_conj}
Whenever $n_r^{\first}$ exists, the set $\Fset_r$ is conjugation-invariant.
\end{proposition}

\begin{proof}
By Corollary~\ref{cor:conj_layers}, every layer $L_r(n)$ is conjugation-invariant. Taking $n=n_r^{\first}$, we obtain the claim.\qedhere
\end{proof}

In particular, every non-self-conjugate first-occurring partition appears together with its conjugate.

\subsection{Finite-range computational observations}

Beyond the initial exact range, we record only finite-range computational observations. Exhaustive enumeration of partitions up to $n=29$ is consistent with the same staircase pattern:
\[
n_r^{\first}=1+\frac{r(r+1)}{2},
\qquad
\Fset_r=\Stair_r,
\qquad
|\Fset_r|=r+1,
\]
for
\[
0\le r\le 7.
\]
At present, however, we record this only as a finite-range computational observation outside the exact range covered by Theorem~\ref{thm:firstocc_initial}. We do not treat these rows as theorem-level results, and we do not reproduce the full computational data in the present paper.

The initial segment of the first-occurrence table is shown in Table~\ref{tab:firstocc}.

\begin{table}[ht]
\centering
\footnotesize
\setlength{\tabcolsep}{4pt}
\begin{tabular}{c c c >{\raggedright\arraybackslash}p{8.1cm}}
\toprule
$r$ & $n_r^{\first}$ & $|\Fset_r|$ & representatives of $\Fset_r$ up to conjugation \\
\midrule
0 & 1  & 1 & $(1)$ \\
1 & 2  & 2 & $(2)$ \\
2 & 4  & 3 & $(3,1),\ (2,2)$ \\
3 & 7  & 4 & $(4,2,1),\ (3,3,1)$ \\
4 & 11 & 5 & $(5,3,2,1),\ (4,4,2,1),\ (4,3,3,1)$ \\
5 & 16 & 6 & $(6,4,3,2,1),\ (5,5,3,2,1),\ (5,4,4,2,1)$ \\
6 & 22 & 7 & $(7,5,4,3,2,1),\allowbreak\ (6,6,4,3,2,1),\allowbreak\ (6,5,5,3,2,1),\allowbreak\ (6,5,4,4,2,1)$ \\
7 & 29 & 8 & $(8,6,5,4,3,2,1),\allowbreak\ (7,7,5,4,3,2,1),\allowbreak\ (7,6,6,4,3,2,1),\allowbreak\ (7,6,5,5,3,2,1)$ \\
\bottomrule
\end{tabular}
\caption{Initial first-occurrence data for simplex layers. Rows $r=0,1,2,3$ are exact; rows $r=4,5,6,7$ record finite-range computational observations consistent with exhaustive enumeration up to $n=29$.}
\label{tab:firstocc}
\end{table}

This table serves three purposes:
\begin{enumerate}[label=(\arabic*)]
\item it records the earliest appearance of each listed layer;
\item it identifies the partitions that realize that layer for the first time;
\item it separates theorem-level initial results from broader finite-range computations.
\end{enumerate}

The computed data suggest that the sets $\Fset_r$ follow a rigid staircase-based pattern. At present we state this only as a computational pattern, not as a theorem.

\subsection{Structural questions}

The first-occurrence table leads naturally to several further questions.

The first is whether the sets $\Fset_r$ always coincide with the staircase families $\Stair_r$. The computed data suggest an affirmative answer, but such observations should remain computational unless supported by a uniform proof.

The second is whether first occurrences are always explained by one of the two pure local mechanisms above, or whether genuinely mixed local configurations already occur at minimal size.

The third is the asymptotic growth of the sequence
\[
r\longmapsto n_r^{\first}.
\]
At present this remains open.

\section{Adjacent-layer phase boundaries}\label{sec:boundaries}

\subsection{Edge, vertex, and one-sided boundaries}

We now study the formal interfaces between adjacent simplex layers. Since the layers
\[
L_r(n)=\{\lambda\vdash n:\dimloc(\lambda)=r\}
\]
are level sets of a vertex invariant, the natural notion of phase boundary is graph-theoretic.

For $n\ge 1$ and $r\ge 0$, define the adjacent-layer edge boundary
\[
\Ebdry_{r,r+1}(n)
:=
\bigl\{
\{\lambda,\mu\}\in E(G_n):
\lambda\in L_r(n),\ \mu\in L_{r+1}(n)
\bigr\}.
\]

Its endpoint set is the adjacent-layer vertex boundary
\[
\Vbdry_{r,r+1}(n)
:=
\bigl\{
\nu\in L_r(n)\cup L_{r+1}(n):
\nu \text{ is incident to an edge of }\Ebdry_{r,r+1}(n)
\bigr\}.
\]

It is also convenient to separate the two sides:
\[
\bminus_{r,r+1}(n):=\Vbdry_{r,r+1}(n)\cap L_r(n),
\qquad
\bplus_{r,r+1}(n):=\Vbdry_{r,r+1}(n)\cap L_{r+1}(n).
\]

These are the lower-side and upper-side boundaries, respectively.

\subsection{Basic properties of adjacent-layer boundaries}

The first formal properties are immediate from the definitions.

\begin{proposition}\label{prop:boundary_basic}
For every $n\ge 1$ and $r\ge 0$:
\begin{enumerate}[label=(\arabic*)]
\item $\Ebdry_{r,r+1}(n)$ is precisely the set of edges joining $L_r(n)$ to $L_{r+1}(n)$;

\item $\bminus_{r,r+1}(n)$ is exactly the set of vertices in $L_r(n)$ having a neighbor in $L_{r+1}(n)$;

\item $\bplus_{r,r+1}(n)$ is exactly the set of vertices in $L_{r+1}(n)$ having a neighbor in $L_r(n)$;

\item one has the disjoint union
   \[
   \Vbdry_{r,r+1}(n)=\bminus_{r,r+1}(n)\sqcup \bplus_{r,r+1}(n).
   \]
\end{enumerate}
\end{proposition}

\begin{proof}
The first statement is the definition of $\Ebdry_{r,r+1}(n)$. The second and third follow by taking endpoints on the corresponding side. The fourth is immediate because the layers $L_r(n)$ and $L_{r+1}(n)$ are disjoint.\qedhere
\end{proof}

\begin{lemma}\label{lem:boundary_empty}
If $L_r(n)=\varnothing$ or $L_{r+1}(n)=\varnothing$, then
\[
\Ebdry_{r,r+1}(n)=\varnothing.
\]
\end{lemma}

\begin{proof}
Immediate from the definition.\qedhere
\end{proof}

Thus the boundary between two adjacent layers is defined purely in graph-theoretic terms, independently of any geometric visualization.

\subsection{Conjugation invariance}

Since the layers themselves are conjugation-invariant, the same holds for the boundaries between them.

\begin{proposition}\label{prop:boundary_conj}
For every $n\ge 1$ and $r\ge 0$, the sets
\[
\Ebdry_{r,r+1}(n),\qquad
\Vbdry_{r,r+1}(n),\qquad
\bminus_{r,r+1}(n),\qquad
\bplus_{r,r+1}(n)
\]
are invariant under partition conjugation.
\end{proposition}

\begin{proof}
Conjugation is an automorphism of the partition graph $G_n$, so it preserves adjacency. By Corollary~\ref{cor:conj_layers} it also preserves each exact layer $L_r(n)$. Hence a cross-layer edge between $L_r(n)$ and $L_{r+1}(n)$ is sent to another such edge, proving invariance of $\Ebdry_{r,r+1}(n)$. The remaining statements follow by taking endpoints and restricting to the corresponding layer.\qedhere
\end{proof}

In particular, every non-self-conjugate boundary edge or boundary vertex occurs together with its conjugate counterpart.

\subsection{Boundary non-emptiness}

Boundary non-emptiness admits the following equivalent formulations.

\begin{proposition}\label{prop:boundary_nonempty}
The following are equivalent:
\begin{enumerate}[label=(\arabic*)]
\item $\Ebdry_{r,r+1}(n)\neq\varnothing$;
\item there exists $\lambda\in L_r(n)$ adjacent to some $\mu\in L_{r+1}(n)$;
\item both $\bminus_{r,r+1}(n)$ and $\bplus_{r,r+1}(n)$ are nonempty.
\end{enumerate}
\end{proposition}

\begin{proof}
The equivalence of (1) and (2) is immediate from the definition. If (1) holds, the endpoints of any edge in $\Ebdry_{r,r+1}(n)$ belong to the two one-sided boundaries, so (3) holds. Conversely, if both one-sided boundaries are nonempty, then by definition each contains a vertex incident to a cross-layer edge, and therefore $\Ebdry_{r,r+1}(n)\neq\varnothing$.\qedhere
\end{proof}

Thus boundary non-emptiness is exactly the statement that the two adjacent exact layers are actually adjacent in the graph.

\begin{remark}
The purpose of the present section is primarily definitional: it fixes the exact interface language associated with adjacent layers. Structural questions about the size, shape, or distribution of these boundaries are left for later work.
\end{remark}

\subsection{Cross-layer edges beyond adjacent layers}

The formal phase boundary defined above concerns only adjacent layer values. In principle, the graph $G_n$ may also contain edges joining $L_r(n)$ to $L_s(n)$ with $|r-s|\ge 2$. To distinguish this broader phenomenon from the present notion of phase boundary, define for $r\neq s$
\[
E_{r,s}(n)
:=
\bigl\{
\{\lambda,\mu\}\in E(G_n):
\lambda\in L_r(n),\ \mu\in L_s(n)
\bigr\}.
\]

Then
\[
\Ebdry_{r,r+1}(n)=E_{r,r+1}(n)
\]
is the special case corresponding to neighboring exact levels.

The adjacent-layer boundary is therefore one part of the full edge-jump structure of $\dimloc$. The latter will be treated as a separate structural question.

\section{Restricted localization and comparison with shell geometry}\label{sec:localization}

\subsection{Restricted layer sets and restricted boundary sets}

We now relate the layer language to broader geometric subsets of the partition graph. Let $X_n\subseteq \Par(n)$ be a distinguished subset, for example the boundary framework, the self-conjugate axis, the spine, the central region, or the rear-central region.

For $r\ge 0$, define the restricted layer set
\[
L_r(n;X_n):=L_r(n)\cap X_n.
\]

Likewise, define the restricted adjacent-layer boundary
\[
\Ebdry_{r,r+1}(n;X_n)
:=
\bigl\{
\{\lambda,\mu\}\in \Ebdry_{r,r+1}(n):
\lambda,\mu\in X_n
\bigr\}.
\]

These restricted sets provide the language for localization questions. For example, the statement
\[
L_r(n)\cap X_n=\varnothing
\]
means that the layer $L_r(n)$ is absent from $X_n$, while
\[
\Ebdry_{r,r+1}(n;X_n)=\varnothing
\]
means that the adjacent-layer interface has no internal edges inside $X_n$.

\subsection{Conjugation-compatible localization}

The restriction formalism interacts well with conjugation-symmetric geometric subsets.

\begin{proposition}\label{prop:restriction_conj}
Let $X_n\subseteq \Par(n)$ be invariant under partition conjugation. Then, for every $r$,
\[
L_r(n;X_n)
\]
is conjugation-invariant. Likewise, for every $r$,
\[
\Ebdry_{r,r+1}(n;X_n)
\]
is conjugation-invariant.
\end{proposition}

\begin{proof}
By Corollary~\ref{cor:conj_layers}, each layer $L_r(n)$ is conjugation-invariant. Hence its intersection with a conjugation-invariant set $X_n$ is again conjugation-invariant. The boundary statement follows similarly from Proposition~\ref{prop:boundary_conj}.\qedhere
\end{proof}

In particular, this applies to any distinguished subset $X_n$ defined in a conjugation-symmetric way.

\subsection{Exact layer language versus shell language}

The simplex layers and adjacent-layer boundaries are exact graph-theoretic objects defined in terms of $\dimloc$. By contrast, the shell language refers to broader geometric regions and thickening regimes. The two descriptions are related but not identical.

In particular, the layer language concerns exact level sets such as
\[
L_r(n),\qquad \toplayer(n),\qquad \Ebdry_{r,r+1}(n),
\]
whereas the shell language refers to broader regions such as the triangular skin, the tetrahedral core, and higher-dimensional rear-central thickening.

The comparison should therefore be restriction-based: one asks how layers and adjacent-layer boundaries are distributed inside those broader geometric regions. One should not identify shells with layers unless this is proved in a precise formal sense.

\subsection{What is theorem-level and what is computational}

The restriction formalism gives the exact language for localization statements, but it does not by itself prove such statements. Stronger claims, for example that higher layers avoid most of the framework or concentrate near the rear-central region, should be treated computationally unless supported by separate arguments.

Accordingly, the role of this section is modest: it fixes the framework in which localization data can be stated. Structural localization results and detailed restricted counts are not developed here.

\section{Natural computational datasets and sequences}\label{sec:atlas}

\subsection{The layer profile array}

We briefly indicate the natural computational objects attached to the stratification. The central object is the layer profile array
\[
a_{n,r}:=|L_r(n)|.
\]

For fixed $n$, the row
\[
(a_{n,r})_{r}
\]
records the exact layer profile of the graph $G_n$. It determines the realized spectrum
\[
\SpecDelta(n)=\{r:a_{n,r}>0\},
\]
the bottom and top local simplex dimensions
\[
\deltaloc(n)=\min\{r:a_{n,r}>0\},
\qquad
\Deltaloc(n)=\max\{r:a_{n,r}>0\},
\]
and the size of the top layer
\[
|\toplayer(n)|=a_{n,\Deltaloc(n)}.
\]

Thus the array $(a_{n,r})$ is the basic computational object attached to the stratification.

\subsection{Top-layer and first-occurrence data}

A second natural dataset concerns the upper end of the stratification. For each $n$, one records
\[
\Deltaloc(n)
\qquad\text{and}\qquad
\tau_{\mathrm{top}}(n):=|\toplayer(n)|.
\]

A third dataset concerns first occurrence. For each $r$, one records
\[
n_r^{\first},
\qquad
|\Fset_r|,
\qquad
\Fset_r\text{ (up to conjugation)}.
\]

Together, these tables describe both the vertical structure of a fixed graph $G_n$ and the horizontal emergence of new layers across the sequence $(G_n)$.

\subsection{Boundary tables}

The phase-boundary theory of Section~\ref{sec:boundaries} leads to a parallel family of interface counts. For each $n$ and $r$, define
\[
b^E_{r,r+1}(n):=\bigl|\Ebdry_{r,r+1}(n)\bigr|,
\]
\[
b^-_{r,r+1}(n):=\bigl|\bminus_{r,r+1}(n)\bigr|,
\qquad
b^+_{r,r+1}(n):=\bigl|\bplus_{r,r+1}(n)\bigr|,
\]
and
\[
b^V_{r,r+1}(n):=\bigl|\Vbdry_{r,r+1}(n)\bigr|.
\]

These quantities measure, respectively, the total amount of adjacency between neighboring layers, the number of participating vertices on the lower and upper sides, and the total width of the interface in vertex terms.

They replace purely visual statements about where one shell appears to pass into another.

\subsection{Restricted regional counts}

To compare layer theory with shell geometry, the same computational framework may also include restricted counts for selected geometric regions $X_n\subseteq \Par(n)$. For example, one may record
\[
a_{n,r}(X_n):=|L_r(n)\cap X_n|
\]
and, for adjacent-layer interfaces,
\[
b^E_{r,r+1}(n;X_n):=
\bigl|\Ebdry_{r,r+1}(n;X_n)\bigr|.
\]

These restricted counts do not by themselves prove localization theorems, but they state the localization data in exact layer language and make it possible to compare different geometric regions on the same footing.

\subsection{Natural sequences and displayed data}

The most natural integer sequences arising from this computational framework are:
\begin{enumerate}[label=(\arabic*)]
\item the top local simplex dimension
   \[
   \Deltaloc(n);
   \]

\item the top-layer size
   \[
   \tau_{\mathrm{top}}(n)=|\toplayer(n)|;
   \]

\item the first-occurrence index
   \[
   n_r^{\first};
   \]

\item the fixed-layer size sequences
   \[
   a_r(n):=|L_r(n)|;
   \]

\item the adjacent-layer boundary sequences
   \[
   b_r(n):=|\Ebdry_{r,r+1}(n)|.
   \]
\end{enumerate}

These sequences should be distinguished clearly from theorem-level results. Their role in the paper is descriptive: they display the concrete content of the stratification and suggest further structural questions.

Possible visual companions include layer-profile plots, first-occurrence plots, and heat maps for boundary counts.

\section{Conclusion and open problems}\label{sec:conclusion}

\subsection{Exact results}

The exact part of the paper establishes the layer stratification induced by $\dimloc$, rewrites layer membership in terms of the local star/top capacities, and proves conjugation invariance. It also introduces the adjacent-layer edge boundary as the interface between neighboring levels, organizes the first-occurrence problem, and isolates the initial layer values for which theorem-level results are obtained.

Thus the local simplex dimension is promoted from a vertex-wise local invariant to a global stratification of the partition graph.

\subsection{Computational outcomes}

The computational component of the paper is limited to the explicit first-occurrence table of Section~\ref{sec:firstocc}, together with the identification of the natural datasets attached to the stratification, namely the quantities
\[
|L_r(n)|,\qquad
\deltaloc(n),\qquad
\Deltaloc(n),\qquad
|\toplayer(n)|,\qquad
n_r^{\first},\qquad
|\Ebdry_{r,r+1}(n)|,
\]
and their restricted counterparts inside selected geometric regions. Apart from the first-occurrence table, these datasets are indicated here mainly as a computational framework for later work. Whenever such data are used, they remain computational unless supported by separate proofs.

In particular, questions about concentration of higher layers, migration of adjacent-layer interfaces, or regularity of top-layer growth remain observational in the present paper.

\subsection{Open problems}

The most natural next questions fall into several groups.

\begin{problem}
Determine the growth of the first-occurrence sequence
\[
r\longmapsto n_r^{\first}.
\]
Is there an exact formula, a structural characterization, or a robust asymptotic law?
\end{problem}

\begin{problem}
Identify the partitions in the first-occurrence sets
\[
\Fset_r.
\]
Do they belong to a recognizable extremal family?
\end{problem}

\begin{problem}
Does the realized layer spectrum
\[
\SpecDelta(n)
\]
always form an interval? Equivalently, can there be gaps between realized layer values for a fixed $n$?
\end{problem}

\begin{problem}
How much can the local simplex dimension change along a single edge of $G_n$? More precisely, what bounds can one place on
\[
|\dimloc(\lambda)-\dimloc(\mu)|
\qquad (\lambda\sim\mu)?
\]
\end{problem}

\begin{problem}
Which computational localization patterns admit exact proofs? In particular, do higher layers avoid substantial parts of the framework, and are there exact rear-central localization theorems for sufficiently large realized layer values?
\end{problem}

\begin{problem}
What can be said about the combinatorial structure of the adjacent-layer interfaces
\[
\Ebdry_{r,r+1}(n)?
\]
Are there exact formulas, sharp bounds, or recurring structural types for these boundaries?
\end{problem}

\subsection{Final remark}

The partition graph continues to display a characteristic duality already visible elsewhere in the project: a simple local move system gives rise to unexpectedly rich global structure. Here that structure appears as a stratification by local simplex dimension, together with canonical interfaces between neighboring levels. The simplex layers and their phase boundaries are, in this sense, a formal skeleton of that geometry.

\section*{Acknowledgements}
The author acknowledges the use of ChatGPT (OpenAI) for discussion, structural planning, and editorial assistance during the preparation of this manuscript. All mathematical statements, proofs, computations, and final wording were checked and approved by the author, who takes full responsibility for the contents of the paper.


\begin{thebibliography}{99}

\bibitem{Andrews}
George E. Andrews,
\emph{The Theory of Partitions},
Encyclopedia of Mathematics and its Applications, Vol.~2, Addison-Wesley, Reading, MA, 1976.

\bibitem{Hatcher}
Allen Hatcher,
\emph{Algebraic Topology},
Cambridge University Press, Cambridge, 2002.

\bibitem{Stanley}
Richard P. Stanley,
\emph{Enumerative Combinatorics. Vol.~1}, second edition,
Cambridge University Press, Cambridge, 2011.

\bibitem{ProjectLocal}
Fedor B. Lyudogovskiy,
\emph{Local Morphology of the Partition Graph},
arXiv:\href{https://arxiv.org/abs/2603.18696}{2603.18696}, 2026.

\bibitem{ProjectShells}
Fedor B. Lyudogovskiy,
\emph{Simplicial shells and thickness in the partition graph},
arXiv:\href{https://arxiv.org/abs/2603.28171}{2603.28171}, 2026.

\bibitem{ProjectAxial}
Fedor B. Lyudogovskiy,
\emph{Axial Morphology of the Partition Graph: Self-Conjugate Axis, Spine, and Concentration},
arXiv:\href{https://arxiv.org/abs/2603.22546}{2603.22546}, 2026.

\bibitem{ProjectBoundary}
Fedor B. Lyudogovskiy,
\emph{Boundary Framework, Rear Morphology, and Rectangular Ears in the Partition Graph},
arXiv:\href{https://arxiv.org/abs/2603.24824}{2603.24824}, 2026.

\bibitem{ProjectDirectional}
Fedor B. Lyudogovskiy,
\emph{Directional Geometry and Anisotropy in the Partition Graph},
arXiv:\href{https://arxiv.org/abs/2603.25488}{2603.25488}, 2026.

\end{thebibliography}
\end{document}